\documentclass[a4, 12pt]{amsart}
\usepackage[mathscr]{eucal}
\usepackage{amssymb}
\usepackage{latexsym}
\usepackage{amsthm}
\theoremstyle{plain}
\newtheorem{theorem}{Theorem}[section]
\newtheorem{corollary}{Corollary}[section]
\newtheorem{remark}{Remark}[section]

\makeatletter
\@addtoreset{equation}{section}

\title[Complete self-shrinkers]
{Complete self-shrinkers \\ of the mean curvature flow*}
\author {Qing-Ming Cheng and Yejuan Peng}
\address{Qing-Ming Cheng \\  \newline \indent Department of Mathematics, Graduate School of Science and Engineering, \newline \indent Saga University, Saga 840-8502, Japan. cheng@ms.saga-u.ac.jp}
\address{Yejuan Peng \\  \newline \indent Department of Mathematics,  Graduate School of Science and Engineering \\  \newline \indent Saga University, Saga 840-8502,  Japan. yejuan666@gmail.com}

\begin{document}
\maketitle

\begin{abstract}
It is our purpose to study complete self-shrinkers in Euclidean space. 
By introducing a generalized maximum
principle for $\mathcal{L}$-operator, we  give estimates on supremum 
and infimum  of  the squared norm of the second fundamental form 
of self-shrinkers without  assumption on \emph{polynomial volume
growth}, which is assumed in Cao and Li \cite{CL1}. Thus, we can obtain
the rigidity theorems on complete self-shrinkers  without  assumption 
on \emph{polynomial volume growth}.
For complete proper self-shrinkers of dimension 2 and 3,
we give a classification of them under assumption of constant
squared norm of the second fundamental form.

\end{abstract}

\footnotetext{* Research partially Supported by a Grant-in-Aid for
Scientific Research from JSPS.}

\section{introduction}
\noindent  
The mean curvature flow is a well known geometric evolution equation.
The study of the mean curvature from the perspective of partial differential
equations commenced with Huisken's paper  \cite{H1}  on the flow of convex
hypersurfaces. Now the study of the mean curvature flow of submanifolds of higher
codimension has started to receive attentions.

\vskip 1mm\noindent 
One of the most important problems in the  mean curvature flow is to understand
the possible singularities that the flow goes through. Singularities are unavoidable
as the flow contracts any closed embedded submanifold in Euclidean space
eventually leading to extinction of the evolving submanifold. A key starting point
for singularity analysis is Huisken's monotonicity formula because the monotonicity
implies that the flow is asymptotically self-similar near a given singularity and 
thus, is modeled  by  self-shrinking solutions of the flow.

\vskip 1mm\noindent 
Let  $X: M^n\to \mathbb{R}^{n+p}$  be an $n$-dimensional submanifold in the $n+p$-dimensional 
Euclidean space $\mathbb{R}^{n+p}$. If the position vector $X$ evolves in the direction of the mean
curvature $H$, then it gives rise  to a solution to the mean curvature flow:
$$
F(\cdot, t):M^n\to  \mathbb{R}^{n+p}
$$
satisfying $F(\cdot, 0)=X(\cdot)$ and 
\begin{equation}
\dfrac{\partial F(p,t)}{\partial t}=H(p,t), \quad (p,t)\in M\times [0,T),
\end{equation}
where $H(p,t)$ denotes the mean curvature vector  of submanifold $M_t=F(M^n,t)$ at point $F(p,t)$.
The equation (1.1) is called the mean curvature flow equation.
A submanifold $X: M^n\to \mathbb{R}^{n+p}$ is said to be a self-shrinker in $\mathbb{R}^{n+p}$ if it satisfies
\begin{equation}
H=-X^N,
\end{equation}
where 
$X^N$ denotes the orthogonal projection of $X$ into the normal
bundle of $M^n$ (cf. Ecker-Huisken \cite{EH}).

\vskip 1mm\noindent 
U. Abresch and J. Langer
\cite {AL} gave a complete classification of all  self-shrinkers of 
dimension one, that is, self-shrinkers are  curve.
These curves are now called Abresch-Langer curves.

\vskip 1mm\noindent 
In the hypersurface case,  Huisken \cite{H2,H3} proved a classification theorem 
that the only possible smooth self-shrinkers $M^n$ in $\mathbb{R}^{n+1}$ with non-negative
mean curvature, bounded $|A|$, and polynomial volume growth are
isometric to $\Gamma\times\mathbb{R}^{n-1}$ or
$S^k(\sqrt k)\times\mathbb{R}^{n-k}(0\leq k\leq n)$. Here, $\Gamma$ is a
Abresch-Langer curve and $S^k(\sqrt k)$ is a $k$-dimensional sphere.
Colding and  Minicozzi \cite{CM1} showed that Huisken's classification theorem 
still holds without the assumption
that $|A|$ is \emph{bounded}. Furthermore, they showed that the only
smooth embedded entropy stable self-shrinkers with polynomial volume
growth in $\mathbb{R}^{n+1}$ are the  hyperplane $\mathbb{R}^{n}$, 
the  sphere $S^n(\sqrt{n})$ and
the cylinders $S^m(\sqrt{m})\times\mathbb{R}^{n-m},\ 1\leq m\leq n-1$.  Kleene-M${\o}$ller \cite{KM}  classified  complete embedded self-shrinkers of
revolution  in $\mathbb{R}^{n+1}$.
Based on an identity of Colding and Minicozzi
(see (9.42) in \cite {CM1}),  Le and Sesum \cite{LS} proved a
gap theorem on the squared norm of the second fundamental form for self-shrinkers of codimension one:

\vskip 2mm \noindent {\bf Theorem A}  (Le and Sesum \cite{LS}).  Let $M^n$ be 
an $n$-dimensional  complete embedded
self-shrinker without boundary and with polynomial volume growth in  
$\mathbb{R}^{n+1}$. If the squared norm  $|A|^2$ of the second
fundamental form satisfies $|A|^2<1$ , then $M^n$ is a hyperplane.

\vskip 1mm\noindent 
In the higher  codimension case, K. Smoczyk in \cite{S1}
proved that let $M^n$ be a complete  self-shrinker with
$H\neq 0$ and  with parallel  principal normal vector $\nu=H/|H|$  in the normal
bundle, if  $M^n$  has uniformly bounded geometry, then 
$M^n$ must be $\Gamma\times \mathbb{R}^{n-1}$ or
$\tilde{M}^r\times  \mathbb{R}^{n-r}$. Here $\Gamma$ is an
Abresch-Langer curve and $\tilde{M}$ is a minimal submanifold in
sphere. Very recently, Li and Wei \cite{LW} have proved this result
in a weaker condition.
Furthermore, Cao and Li \cite{CL1} extended  the 
classification theorem  for self-shrinkers in Le and Sesum \cite{LS} to  arbitrary codimension,
and proved the following

\vskip 2mm\noindent {\bf Theorem B} (Cao and Li \cite{CL1}). 
Let $M^n$ be 
an $n$-dimensional  complete 
self-shrinker without boundary and with polynomial volume growth in  
$\mathbb{R}^{n+p} \ (p\geq 1)$. If the squared norm  $|A|^2$ of the second
fundamental form satisfies $|A|^2\leq 1$,
then $M^n$ is one of the followings:
\begin{enumerate}
\item a round sphere $S^n(\sqrt{n})$
in $\mathbb{R}^{n+1}$,
\item a cylinder
$S^m(\sqrt{m})\times\mathbb{R}^{n-m},\quad 1\leq m\leq n-1$
in $\mathbb{R}^{n+1}$,
\item a hyperplane in $\mathbb{R}^{n+1}$.
\end{enumerate}

\vskip 1mm \noindent 
We should remark that, in proofs of the above theorems for complete and non-compact self-shrinkers,
integral formulas are exploited as a main method. In order to guarantee that the integration by part holds,
the condition of \emph{polynomial volume growth} plays a very important role. 
Moreover,  Cao and Li \cite {CL1} have asked  whether it is possible to
remove the assumption on  \emph{polynomial volume growth} in their theorem.

\vskip 1mm \noindent \noindent
In this paper,  our purpose is to study  complete self-shrinkers without 
the assumption on \emph{polynomial volume growth}. In order to do
it, we extend the generalized maximum principle of Yau to 
$\mathcal{L}$-operator (see Theorem 3.1). By making use of 
the generalized maximum principle for $\mathcal{L}$-operator, 
we prove the following:
\begin{theorem}
Let $X:M^n\rightarrow\mathbb{R}^{n+p} \ (p\geq 1)$ be an
$n$-dimensional complete self-shrinker without boundary in
$\mathbb{R}^{n+p}$, then one of the following holds:
\begin{enumerate}
\item $\sup |A|\geq 1$,
\item $|A|\equiv 0$, i.e.
$M^n$ is a hyperplane in $\mathbb{R}^{n+1}$.
\end{enumerate}
\end{theorem}

\begin{corollary}
Let $X:M^n\rightarrow \mathbb{R}^{n+p} \ (p\geq 1)$ be a complete
self-shrinker without boundary, and satisfy
$$ \sup |A|^2< 1.$$
Then $M$ is a hyperplane in $\mathbb{R}^{n+1}.$
\end{corollary}

\begin{remark}
The round sphere $S^n(\sqrt{n})$  and the cylinder
$S^k(\sqrt{k})\times \mathbb{R}^{n-k}, \, 1\leq k \leq n-1$ are
complete self-shrinkers in $\mathbb{R}^{n+1}$ with $|A|=1$. Thus, our
result  is sharp.
\end{remark}

\begin{theorem}
Let $X:M^n\rightarrow \mathbb{R}^{n+1} $ be a complete
self-shrinker without boundary. If  \ $\inf H^2>0$ and $|A|^2$ is bounded, 
then $ \inf |A|^2\leq  1$. 
\end{theorem}

\begin{corollary}
Let $X:M^n\rightarrow \mathbb{R}^{n+1} $ be a complete
self-shrinker without boundary. If \ $\inf H^2>0$ and $|A|^2$ is constant, 
then $ |A|^2\equiv 1$ and $M^n$ is  the round sphere $S^n(\sqrt{n})$  
or the cylinder $S^k(\sqrt{k})\times \mathbb{R}^{n-k}, \, 1\leq k \leq n-1$. 
\end{corollary}
\begin{remark} 
In  \cite{CL1, CM1, H2, H3} and so on, they assume that $M^n$ has polynomial volume growth.
In our results,  we do not assume   the condition on polynomial volume growth.
We should notice that condition $\inf H^2>0$ is necessary since Angenent \cite{A}  has proved that there exist
embedded self-shrinkers from $S^1\times S^{n-1}$ into  $\mathbb{R}^{n+1} $ with $\inf H^2=0$ {\rm(}cf. \cite{KM}{\rm)}.
\end{remark}

\vskip 2mm
\noindent
In  section 4, we shall consider complete proper self-shrinkers
of 2 and 3 dimensions. We will try to classify complete proper self-shrinkers
of 2 and 3 dimensions under condition that  the squared norm of 
the second fundamental form is constant.

\vskip 2mm
\noindent
{\bf Acknowledgements.}  We would like to express our gratitude to 
Professor H. Li  and Dr.  Y.  Wei for many helpful discussions on self-shrinkers
and for sharing their results in \cite {LW} with us prior to publication in arXiv.
We wish to thank professor G.  Wei for useful suggestions.
  
\section{Preliminaries}

\noindent Let $X: M^n\rightarrow\mathbb{R}^{n+p}$ be an
n-dimensional connected submanifold of the (n+p)-dimensional Euclidean space
$\mathbb{R}^{n+p}$. We choose a local orthonormal frame field
$\{e_A\}_{A=1}^{n+p}$ in $\mathbb{R}^{n+p}$ with dual coframe field 
$\{\omega_A\}_{A=1}^{n+p}$, such that, restricted to $M^n$,
$e_1,\cdots, e_n$ are tangent to $M^n$. The following
conventions on the ranges of indices are used in this paper:
$$
1\leq A,B,C,D\leq n+p, \quad 1\leq i,j,k,l\leq n, \quad n+1\leq
\alpha,\beta,\gamma\leq n+p.
$$
Then we have 
\begin{equation*}
dX=\sum_i\limits \omega_i e_i, \quad de_i=\sum_j\limits \omega_{ij}e_j+\sum_{\alpha}\limits
\omega_{i\alpha}e_\alpha
\end{equation*}
and 
\begin{equation*}
de_\alpha=\sum_i\limits\omega_{\alpha i}e_i+\sum_\beta\limits
\omega_{\alpha\beta}e_\beta.
\end{equation*}
We restrict these forms to $M^n$, then
\begin{equation}
\omega_\alpha=0 \quad \text{for}\quad  n+1\leq\alpha\leq n+p
\end{equation}
and the induced Riemannian metric of $M^n$ is written as
$ds^2_M=\sum_i\limits\omega^2_i$.
From (2.1) and Cartan's lemma, we get 
\begin{equation*}
\omega_{i\alpha}=\sum_j h^\alpha_{ij}\omega_j,\quad
h^\alpha_{ij}=h^\alpha_{ji}.
\end{equation*}
The induced structure equations of $M^n$ are given by 
\begin{equation*}
d\omega_{i}=\sum_j \omega_{ij}\wedge\omega_j, \quad  \omega_{ij}=-\omega_{ji},
\end{equation*}
\begin{equation*}
d\omega_{ij}=\sum_k \omega_{ik}\wedge\omega_{kj}-\frac12\sum_{k,l}
R_{ijkl} \omega_{k}\wedge\omega_{l},
\end{equation*}
where 
\begin{equation}
R_{ijkl}=\sum_\alpha\left(h^\alpha_{ik}h^\alpha_{jl}-h^\alpha_{il}h^\alpha_{jk}\right)
\end{equation}
denotes components of the curvature tensor of $M^n$.
The second fundamental form and the mean curvature vector field of $M^n$
are given by 
$$
A=\sum_{\alpha,i,j}
h^\alpha_{ij}\omega_i\otimes\omega_j\otimes e_\alpha
$$ 
and 
$$
\mathbf{H}=\sum_\alpha\limits
H^\alpha e_\alpha=\sum_\alpha\limits \sum_i\limits h^\alpha_{ii}e_\alpha,
$$
respectively.
Let  
$|A|^2=\sum_{\alpha,i,j}\limits(h^\alpha_{ij})^2$ be  the  squared norm
of the second fundamental form and $H=|\mathbf{H}|$ denote  the mean
curvature of $M^n$.
From (2.2),  components of the Ricci curvature  of $M^n$ are given by 
\begin{equation}
R_{ik}=\sum_\alpha H^\alpha
h^\alpha_{ik}-\sum_{\alpha,j}h^\alpha_{ij}h^\alpha_{jk}.
\end{equation}
Let  $R_{\alpha\beta ij}$ denote components of  the normal curvature tensor  
in the normal bundle. We have  Ricci equations:
\begin{equation}
R_{\alpha\beta
kl}=\sum_i\left(h^\alpha_{ik}h^\beta_{il}-h^\alpha_{il}h^\beta_{ik}\right).
\end{equation}
Defining the covariant derivative of $h^\alpha_{ij}$ by
\begin{equation}
\sum_{k}h^\alpha_{ijk}\omega_k=dh^\alpha_{ij}+\sum_k
h^\alpha_{ik}\omega_{kj} +\sum_k h^\alpha_{kj}\omega_{ki}+\sum_\beta
h^\beta_{ij}\omega_{\beta\alpha},
\end{equation}
we obtain the Codazzi equations
\begin{equation}
h_{ijk}^\alpha=h_{ikj}^\alpha.
\end{equation}
By taking exterior differentiation of (2.5), and defining
\begin{equation}
\sum_lh^\alpha_{ijkl}\omega_l=dh^\alpha_{ijk}+\sum_lh^\alpha_{ljk}\omega_{li}
+\sum_lh^\alpha_{ilk}\omega_{lj}+\sum_l
h^\alpha_{ijl}\omega_{lk}+\sum_\beta
h^\beta_{ijk}\omega_{\beta\alpha},
\end{equation}
we have the following Ricci identities:
\begin{equation}
h^{\alpha}_{ijkl}-h^\alpha_{ijlk}=\sum_m
h^\alpha_{mj}R_{mikl}+\sum_m h^\alpha_{im}R_{mjkl}+\sum_\beta
h^\beta_{ij}R_{\beta\alpha kl}.
\end{equation}
Let $f$ be a smooth function on $M^n$, we define the
covariant derivatives $f_i,\, f_{ij},$ and the Laplacian of $f$ as
follows
\begin{equation*}
df=\sum_i f_i \omega_i,\quad \sum_j
f_{ij}\omega_j=df_i+\sum_jf_j\omega_{ji},\quad \Delta f=\sum_i
f_{ii}.
\end{equation*}
The first and second covariant derivatives  of the mean curvature vector field
$\mathbf{H}$
are defined by 
\begin{equation*}
\sum_i H^\alpha_{,i}\omega_i=dH^\alpha+\sum_\beta
H^\beta\omega_{\beta\alpha},
\end{equation*}
\begin{equation*}
\sum_j H^\alpha_{,ij}\omega_j=dH^\alpha_{,i}+ \sum_j H^\alpha_{,j}\omega_{ji}+\sum_\beta
H^\beta_{,i}\omega_{\beta\alpha}.
\end{equation*}
The following elliptic operator
$\mathcal{L}$ introduced by Colding and Minicozzi  in \cite{CM1}
will play a very important role in this paper:
\begin{equation}
\mathcal{L}f=\Delta f-\langle X,\nabla f\rangle
\end{equation}
where $\Delta$ and $\nabla$ denote the Laplacian and the gradient
operator on the self-shrinker, respectively and $\langle
 \cdot,\cdot\rangle$ denotes the standard inner product of
$\mathbb{R}^{n+p}$.  In \cite {CP}, we have studied  eigenvalues
of the $\mathcal{L}$ operator. The sharp universal estimates for 
eigenvalues of the $\mathcal{L}$ operator on compact self-shrinkers
are obtained.

\vskip 1cm
\section{Proof of main results}
\noindent 
In order to  prove  our results, first of all, we prove  the following
generalized maximum principle for $\mathcal{L}$-operator on self-shrinkers:

\begin{theorem} {\rm(}Generalized maximum principle for $\mathcal{L}$-operator {\rm)}
Let $X: M^n\to \mathbb{R}^{n+p}$ {\rm (}$p\geq 1${\rm)} be a complete self-shrinker with Ricci
curvature bounded from below. Let $f$ be any $C^2$-function bounded
from above on this self-shrinker. Then, there exists a sequence of points
$\{p_k\}\subset M^n$, such that
\begin{equation}
\lim_{k\rightarrow\infty} f(X(p_k))=sup f,\quad
\lim_{k\rightarrow\infty} |\nabla f|(X(p_k))=0,\quad
\limsup_{k\rightarrow\infty}\mathcal{L} f(X(p_k))\leq 0.
\end{equation}
\end{theorem}

\begin{proof}
\noindent Since this self-shrinker is a complete Riemannian manifold with Ricci
curvature bounded from below and $f$ is a $C^2$-function bounded
from above on it, by the generalized maximum principle of Yau
in \cite{CY}, then, 
there is a sequence of points $\{p_k\}\subset M^n$, such that 
$$
\lim_{k\rightarrow\infty} f(X(p_k))=sup f,
$$
\begin{equation}
\lim_{k\rightarrow\infty}\limits |\nabla f|(X(p_k))
=\lim_{k\rightarrow\infty}\limits
\frac{2(f(X(p_k))-f(X(p_0))+1)\gamma(p_k)}{k(\gamma^2(p_k)+2) \log(\gamma^2(p_k)+2)}=0,
\end{equation}
and
\begin{equation}
\aligned \limsup_{k\rightarrow\infty} \Delta f(X(p_k))
&\leq 0,
\endaligned
\end{equation}
where $\gamma(p)$ denotes the length of the geodesic from a fixed point $X(p_0)$
to $X(p)$.
Since $X$ is the position vector, then, we have
$$
|X(p_k)|\leq\gamma(p_k)+|X(p_0)|
$$
By Cauchy-Schwarz inequality, we have
$$
\aligned 
&|\langle X(p_k),\nabla f(X(p_k)) \rangle |\leq  |\nabla f(X(p_k))|\cdot |X(p_k)| \\
&=\frac{2(f(X(p_k))-f(X(p_0))+1)\gamma(p_k)}{k(\gamma^2(p_k)+2)log(\gamma^2(p_k)+2)}\cdot |X(p_k)|\\
&\leq\frac{2(f(X(p_k))-f(X(p_0))+1)\gamma(p_k)
(\gamma(p_k)+|X(p_0)|)}{k(\gamma^2(p_k)+2)\log(\gamma^2(p_k)+2)}\\
&\leq \frac{2(f(X(p_k))-f(X(p_0))+1)}
{k \ \log(\gamma^2(p_k)+2)}+\frac{2(f(X(p_k))-f(X(p_0))+1)\gamma(p_k)|X(p_0)|}
{k (\gamma^2(p_k)+2) \log(\gamma^2(p_k)+2)}.
\endaligned
$$
According to (3.2) and the above inequality,  we have
$$
\lim_{k\rightarrow\infty}\limits| \langle X(p_k),\nabla f(X(p_k))\rangle|=0.
$$
Since $\mathcal{L}f=\Delta f- \langle X,\nabla f \rangle$, 
the above formula and (3.3) imply
$$
\limsup_{k\rightarrow\infty}\limits\mathcal{L}f(X(p_k))\leq 0.
$$
\end{proof}

\noindent Now we prove the theorem  1.1 as follows:

\vskip 2mm\noindent
\emph{Proof of Theorem 1.1.} 
Since $M^n$ is a
complete self-shrinker, the self-shrinker equation (1.2) is
equivalent to
\begin{equation}
H^\alpha=- \langle X,e_\alpha \rangle ,\quad n+1\leq\alpha\leq n+p.
\end{equation}
Taking covariant derivative of (3.4) with respect to
$e_i$, we have
\begin{equation}
H^\alpha_{,i}=\sum_k h^\alpha_{ik}\langle X,e_k \rangle ,\quad 1\leq
i\leq n,\quad n+1\leq\alpha\leq n+p.
\end{equation}
Furthermore, by taking covariant derivative of (3.5) with respect to
$e_j$, we have
\begin{equation}
\aligned H^\alpha_{,ij}&=\sum_k h^\alpha_{ikj}\langle
 X,e_k \rangle +h^\alpha_{ij}+\sum_{\beta,k}
h^\alpha_{ik}h^\beta_{kj}\langle X,e_\beta\rangle \\
&=\sum_kh^\alpha_{ikj}\langle X,e_k \rangle
+h^\alpha_{ij}-\sum_{\beta,k} H^\beta h^\alpha_{ik}h^\beta_{kj},
\endaligned
\end{equation}
According to  (3.6), we obtain 
\begin{equation}
\mathcal{L}|H|^2
=2|\nabla H|^2+2|H|^2-2\sum_{\alpha,\beta,i,k}H^\alpha H^\beta
h^\alpha_{ik}h^\beta_{ik}.
\end{equation}
In fact, 
\begin{equation*}
\aligned 
&\mathcal{L}|H|^2
=\Delta |H|^2-\langle x,\nabla |H|^2\rangle \\
&=2|\nabla H|^2+2\sum_{\alpha,i} H^\alpha H^{\alpha}_{,ii}
-2\sum_{\alpha,k}H^\alpha H^\alpha_{,k}\langle X,e_k \rangle \\
&=2|\nabla H|^2-2\sum_{\alpha,k}H^\alpha H^\alpha_{,k}\langle X,e_k \rangle\\
&+2\sum_\alpha H^\alpha\biggl(\sum_kH^\alpha_{,k}\langle
X,e_k \rangle +H^\alpha-\sum_{\beta,i,k}H^\beta h^\alpha_{ik}h^\beta_{ik}\biggl)
 \\
&=2|\nabla H|^2+2|H|^2-2\sum_{\alpha,\beta,i,k}H^\alpha H^\beta
h^\alpha_{ik}h^\beta_{ik}.
\endaligned
\end{equation*}
By the Cauchy-Schwarz inequality, we have
$$
|\sum_{\alpha,\beta,i,k}H^\alpha H^\beta
h^\alpha_{ik}h^\beta_{ik}|\leq |A|^2|H|^2.
$$
Hence, from (3.7) and the above inequality, we get 
\begin{equation}
\mathcal{L}|H|^2\geq 2|\nabla H|^2+2(1-|A|^2)|H|^2.
\end{equation}
\vskip2mm
\noindent
If $\sup |A|^2\geq 1$, there is  nothing to do.
From now, we assume that $\sup |A|^2< 1$.
Thus, 
$\sum_{\alpha,i,j}\limits (h^{\alpha}_{ij})^2 <1$. 
Together with
(2.3), it is easily seen that Ricci curvature is bounded from below.
Since $\frac{|H|^2}n\leq |A|^2 <1$ and by applying the generalized maximum principle for $\mathcal{L}$-operator
to the function $H^2$, 
we have, from (3.8)
$$
0\geq \limsup\mathcal{L}|H|^2\geq 2(1-\sup |A|^2) \sup |H|^2.
$$
Hence,
from $\sup |A|< 1$,  we have $\sup |H|^2=0$, that is,  $H \equiv 0$.
$M^n$ is totally geodesic. From (1.2), we know that $M^n$ is a smooth minimal cone. 
Hence, $M^n$ is  a  hyperplane and $|A|\equiv 0$.
\begin{flushright} 
$\square$
\end{flushright}

\vskip2mm
\noindent
{\it Proof of Theorem 1.2.} Since $|A|^2$ is bounded, we know that $H$ is bounded and 
the Ricci curvature is bounded from below by (2.3).
Without loss of generality, we can assume that $\inf H>0$ according 
to  $\inf H^2>0$. By a direct computation, we have 
$$
\mathcal LH=(1-|A|^2)H.
$$
Applying the generalized maximum principle for $\mathcal{L}$-operator to $-H$, we obtain
$$
0\leq (1-\inf |A|^2)\inf H.
$$
Since $\inf H>0$, we have $\inf |A|^2\leq 1$.
This finishes the proof of the theorem 1.2.
\begin{flushright} 
$\square$
\end{flushright}

\vskip 2mm
\noindent
{\it Proof of Corollary 1.2.}  
According to the theorem 1.2, we have $\inf |A|^2\leq 1$.
Since $H\neq 0$, we know that $M^n$ is not totally geodesic.
According to the theorem 1.1, we know $\sup |A|^2\geq 1$.
Since  $|A|^2$ is constant, we obtain 
 $|A|^2\equiv1$.
Since the codimension of $M^n$ is one,  we have
\begin{equation}
\dfrac12\mathcal L|A|^2=|\nabla A|^2+|A|^2(1-|A|^2).
\end{equation}
Indeed, since
\begin{equation*}
\begin{aligned}
&h^{n+1}_{ijkk}=h^{n+1}_{kkij}+\sum_mh^{n+1}_{mi}R_{mkjk}+\sum_m h^{n+1}_{km}R_{mijk},\\
\end{aligned}
\end{equation*}
we have 
\begin{equation*}
\begin{aligned}
&\Delta h^{n+1}_{ij}=\sum_kh^{n+1}_{kkij}+\sum_{m,k}h^{n+1}_{mi}R_{mkjk}+\sum_{m,k} h^{n+1}_{km}R_{mijk}\\
&=H_{,ij}+H\sum_{k}h^{n+1}_{ki}h^{n+1}_{kj}-|A|^2h^{n+1}_{ij}\\
&=\sum_kh^{n+1}_{ikj}\langle X,e_k \rangle+h^{n+1}_{ij}-|A|^2h^{n+1}_{ij}.\\
\end{aligned}
\end{equation*}
Hence, we have 
\begin{equation}
\begin{aligned}
&\mathcal L h^{n+1}_{ij}
=(1-|A|^2)h^{n+1}_{ij}.\\
\end{aligned}
\end{equation}
From (3.10), we infer
\begin{equation*}
\dfrac12\mathcal L|A|^2=|\nabla A|^2+|A|^2(1-|A|^2).
\end{equation*}
Therefore, from (3.9), we obtain $|\nabla A|^2\equiv 0$ since $ |A|^2\equiv 1$.
Namely, the second fundamental form of $M^n$ is parallel. 
According  to the theorem of Lawson \cite{L},  we know that 
$M^n$ is isometric to the round sphere $S^n(\sqrt{n})$  
or the cylinder $S^k(\sqrt{k})\times \mathbb{R}^{n-k}, \, 1\leq k \leq n-1$. 
\begin{flushright} 
$\square$
\end{flushright}
\vskip 5mm

\section{Self-shrinkers of dimension two and  three}

\noindent
In this section, we assume that  $X: M^n\to \mathbb{R}^{n+p}$ is a complete proper self-shrinker
with $n=2$ or $n=3$.  In \cite{DX2}, Ding and Xin have proved that a two dimensional  complete 
proper self-shrinker in $\mathbb {R}^3$ is a plane, a sphere or a cylinder.  By the theorems of Li and Wei
\cite{LW} and a simple observation, we can prove the following. 
\begin{theorem}
Let $X: M^3\to \mathbb{R}^{5}$ be a 3-dimensional complete proper self-shrinker without boundary
and with $H>0$.  If  the principal normal  $\nu =\frac{\mathbf{H}}{H}$ is parallel in the normal bundle of $M^3$
and the squared norm of the second fundamental form is constant, then $M^3$ is one of the following:
\begin{enumerate}
\item $S^k({\sqrt k})\times \mathbb{R}^{3-k}$, $1\leq k\leq  3$ with $|A|^2=1$,
\item  $S^1(1)\times S^1(1)\times \mathbb{R}^{}$ with $|A|^2=2$,
\item  $S^1(1)\times S^2(\sqrt2)$ with $|A|^2=2$,
\item  the three dimensional minimal  isoparametric Cartan hypersurface with $|A|^2=3$.
\end{enumerate}
\end{theorem}
\begin{proof}
Since $M^3$ is a complete proper self-shrinker, we know that $M^3$
has polynomial volume growth from the result of Ding and Xin \cite{DX1} or 
X. Cheng and Zhou \cite{CZ}. 
Thus, from the theorem 1.1 of Li and Wei \cite{LW}, we know that $M^3$ is isometric
to $\Gamma\times \mathbb{R}^{2}$ or
$\tilde{M}^r\times  \mathbb{R}^{3-r}$, where  $\Gamma$ is an
Abresch-Langer curve and $\tilde{M}$ is a compact  minimal hypersurface in
sphere $S^{r+1}(\sqrt r)$. 

\noindent
Since $|A|^2$ is constant, then the 
Abresch-Langer curve $\Gamma$  must be a circle. In this case,
$M^3$ is isometric to $S^1({1})\times \mathbb{R}^{2}$. 

\noindent
If $|A|^2\leq 1$, from the results of Cao and Li \cite{CL1}, we have 
$|A|^2=1$ and $M^3$ is  $S^k({\sqrt k})\times \mathbb{R}^{3-k}$, $1\leq k\leq  3$.
Hence, we can only consider the case of $|A|^2>1$.

\noindent
When $r=2$,
$\tilde{M}$ is a compact minimal surface in
sphere $S^{3}(\sqrt 2)$ with the squared norm of the second fundamental form
$|\tilde A|^2=|A|^2-1$. Thus, $\tilde{M}$ is the Clifford torus  $S^1(1)\times S^1(1)$
in $S^{3}(\sqrt 2)$. 

\noindent
When $r=3$,
$\tilde{M}$ is a compact minimal hypersurface in
sphere $S^{4}(\sqrt 3)$ with a constant squared norm of the second fundamental form, that is, 
$|\tilde A|^2=|A|^2-1$. Thus, $\tilde{M}$ is the Clifford torus $S^1(1)\times S^2(\sqrt2)$ in 
 $S^{4}(\sqrt 3)$  with $|A|^2=2$ or 
 the three dimensional minimal  isoparametric Cartan hypersurface in  $S^{4}(\sqrt 3)$  with $|A|^2=3$
according to the solution of Chern's conjecture for $n=3$ in \cite{C1}. This finishes the proof of the 
theorem 4.1.
\end{proof}
\begin{theorem}
Let $X: M^2\to \mathbb{R}^{2+p} $  {\rm (}$p\geq 1${\rm)} be a 2-dimensional complete proper self-shrinker without boundary
and with $H>0$.  If  the principal normal  $\nu =\frac{\mathbf{H}}{H}$ is parallel in the normal bundle of $M^2$
and the squared norm of the second fundamental form is constant, then $M^2$ is one of the following:
\begin{enumerate}
\item $S^k({\sqrt k})\times \mathbb{R}^{2-k}$, $1\leq k\leq  2$ with $|A|^2=1$,
\item the  Boruvka sphere $S^2({\sqrt {m(m+1)}})$ in $S^{2m}({\sqrt 2})$ with $p=2m-1$ and 
$|A|^2=2-\frac{2}{m(m+1)}$,
\item  a compact flat minimal surface in $S^{2m+1}(\sqrt2)$ with $p=2m$ and $|A|^2=2$.
\end{enumerate}
\end{theorem}
\begin{proof}
Since $M^2$ is a complete proper self-shrinker, we know that $M^2$
has polynomial volume growth from the result of Ding and Xin \cite{DX1}  or X. Cheng and Zhou \cite{CZ}. 
Thus, from the theorem 1.1 of Li and Wei \cite{LW}, we know that $M^2$ is isometric
to $\Gamma\times \mathbb{R}^{1}$ or
$\tilde{M}^2$, where  $\Gamma$ is an
Abresch-Langer curve and $\tilde{M}$ is a compact  minimal surface in
sphere $S^{p+1}(\sqrt 2)$. 

\noindent
Since $|A|^2$ is constant, then the 
Abresch-Langer curve $\Gamma$  must be a circle. In this case,
$M^2$ is isometric to $S^1({1})\times \mathbb{R}$. 

\noindent
If $|A|^2\leq 1$, from the results of Cao and Li \cite{CL1}, we have 
$|A|^2=1$ and $M^2$ is  $S^k({\sqrt k})\times \mathbb{R}^{2-k}$, $1\leq k\leq  2$.
Hence, we can only consider the case of $|A|^2>1$.

\noindent
Since $\tilde{M}$ is a compact minimal surface in
sphere $S^{p+1}(\sqrt 2)$ with a constant squared norm of the second fundamental form,
that is, $|\tilde A|^2=|A|^2-1$. Thus, $\tilde{M}$ is a compact minimal surface in
sphere $S^{p+1}(\sqrt 2)$ with constant Gauss curvature. According to the classification
of minimal surface in sphere $S^{p+1}(\sqrt 2)$ with constant Gauss curvature due to 
Bryant \cite{B} (cf. Calabi \cite{C}, Kenmotsu \cite{K} and Wallach \cite{W}),
we know that $M^2$ is isometric to 
a  Boruvka sphere $S^2({\sqrt {m(m+1)}})$ in $S^{2m}({\sqrt 2})$ with $p=2m-1$ and 
$|A|^2=2-\frac{2}{m(m+1)}$
or  a compact flat minimal surface in $S^{2m+1}(\sqrt2)$ with $p=2m$ and $|A|^2=2$.
This finishes the proof of the  theorem 4.2.
\end{proof}

\end{document}